\documentclass[a4paper, 12pt]{article}
\usepackage{anysize}
\marginsize{3,5cm}{2,5cm}{2,5cm}{2,5cm}
\usepackage{amssymb}
\usepackage{amsmath,amsbsy,amssymb,amscd}
\usepackage{t1enc}
\usepackage[cp1250]{inputenc}
\usepackage[british]{babel}
\usepackage[all]{xy}
\usepackage{color}
\usepackage{amsfonts}
\usepackage{latexsym}
\usepackage{amsthm}
\usepackage{mathrsfs}

\newcommand{\margem}[1]{\marginpar{{\scriptsize {#1}}}}

\usepackage{mathrsfs} 

\DeclareMathAlphabet{\mathpzc}{OT1}{pzc}{L}{it} 

\newtheorem{definition}{Definition}

\newtheorem{proposition}[definition]{Proposition}
\newtheorem{theorem}[definition]{Theorem}

\newtheorem{remark}[definition]{Remark}
\newtheorem{lemma}[definition]{Lemma}

\def\geq{\geqslant}
\def\leq{\leqslant}
\def\R{\mathbb{R}}
\def\T{\mathbb{T}}

\def\N{\mathbb{N}}

\def\Q{\mathbb{Q}}
\def\th{\theta}
\newcommand{\bea}{\begin{eqnarray}}
  \newcommand{\eea}{\end{eqnarray}}
  \newcommand{\beab}{\begin{eqnarray*}}
  \newcommand{\eeab}{\end{eqnarray*}}

  \newcommand{\be}{\begin{equation}}
  \newcommand{\ee}{\end{equation}}
\title{Rigidity times for weakly mixing dynamical system which are not  rigidity times for any irrational rotation}
\author{Bassam Fayad and Adam Kanigowski}
\begin{document}
\baselineskip=14pt \maketitle

\begin{abstract} We construct an increasing sequence of natural numbers $(m_n)_{n=1}^{+\infty}$ with the property that $(m_n \th [1])_{n\geq 1}$ is dense in $\T$  for any $\th \in \R\setminus \Q$, and a continuous measure on the circle $\mu$ such that $\lim_{n\to +\infty}\int_{\T}\|m_n\theta\|d\mu(\theta)=0$. Moreover, for every fixed $k\in \N$,  the set $\{n\in \N:\,k\nmid m_n \}$ is infinite. 

This is a sufficient condition for the existence of a rigid, weakly mixing dynamical system whose rigidity time is not a rigidity time for any system with a discrete part in its spectrum.
\end{abstract}

\section{Introduction}
Let $\T$ denote the circle group with addition $mod1$. For $\eta\in \R$ we denote by $\eta[1]$ the fractional part of $\eta$ and $\|\eta\|$ its distance to integers. It follows that $\|\eta\|=\min(\eta[1],(1-\eta)[1])$. Therefore for any $\eta\in \R$, $\|\eta\|\leq \frac{1}{2}$. 

In this note, we prove the following two results.
\begin{theorem}\label{ext} Fix rationally independent numbers $\{\alpha_i\}_{i\in \N}\in \T.$\footnote{By this we mean that every finite collection is rationally independent.} There exists an increasing sequence $(m_n)_{n=1}^{+\infty}$ such that
$(m_n \th [1])_{n\geq 1}$ is dense in $\T$  for every irrational $\theta$, and for every $\epsilon>0$ and $k\in \N$ there exists $N_0\in \N$ such that for every $n\geq N_0$ we have   $\|m_n\alpha_i\|<\epsilon$ for at least $k-1$ choices of $i\in \{1,...,k\}$. Moreover for every $k\in \N$ the set $\{n\in \N:\,k\nmid m_n \}$ is infinite.
\end{theorem}

\begin{theorem}\label{mes} Fix rationally independent numbers $\{\alpha_i\}_{i\in \N}\in \T$ and let $(m_n)_{n\geq 1}$ be the corresponding sequence from Theorem \ref{ext}. There exists a continuous probability measure $\mu$ on $\T$ such that $$\lim_{n\to +\infty}\int_{\T}\|m_n\theta\|d\mu(\theta)=0.$$
\end{theorem}

Theorem \ref{ext} gives us an increasing sequence of natural numbers $(m_n)_{n=1}^{+\infty}$ which is not a rigidity time for any system with a discrete part in its spectrum. Indeed, if the system has an irrational eigenvalue then it has the irrational rotation as a factor. If it has a rational eigenvalue then it has a shift on a finite group as a factor.  But a rigidity time for a dynamical system is also a rigidity time for its factors, and a sequence as in Theorem \ref{ext} cannot be a rigidity sequence for any rational or  irrational rotation.

 From Theorem \ref{mes}, by the Gaussian measure space construction (see \cite{Co-Fo-Si}), we deduce that there exists a weakly mixing dynamical system whose rigidity times contain  the constructed sequence $(m_n)_{n=1}^{+\infty}$. This gives a full answer to the question stated in \cite{BJLR} of whether a rigidity times sequence of  a system with discrete spectrum is a rigidity time for some weakly mixing and conversely whether  a rigidity times sequence of  a system with continuous spectrum is a rigidity times sequence  for some discrete spectrum system. The first direction was established in \cite{Ad}  and later in \cite{Fa-Tho}, namely, any rigidity time of a system with discrete spectrum is also a rigidity time for some weakly mixing dynamical system.

\indent Our approach is inspired by the completely spectral approach adopted in  \cite{Fa-Tho}. First we prove the existence of a sequence $m_n$ which is not a rigidity time for any circle rotation, but still satisfies that $\| m_n \alpha_i \|$ is small  for most of the indices $i$ of a family of   rationally independent numbers $\{\alpha_i\}_{i\in \N}\in \T$ (see precise statement in Theorem \ref{ext}).

 This allows to construct a continuous probability measure on $\T$, that is a weak limit of discrete measures each supported on some finite set connected with the numbers $\alpha_1,\alpha_2,\ldots$, with a Fourier transform converging to $1$ along this sequence.  \\
\indent The auhors would like to thank to Jean-Paul Thouvenot  for his meaningful input in solving this problem. 

\section{Proof of Theorem \ref{ext}}

Let there be given a family of rationally independent numbers $\{\alpha_i\}_{i\in \N}\in \T$.
We will first state a lemma, which is a generalisation of Lemma 1 in \cite{Fa-Tho}.
\begin{definition}\em{\cite{Fa-Tho} For an interval $I\subset \T$ and fixed $\epsilon>0$ one says that $\theta\in \mathcal{A}(N_1,N_2,\epsilon,I,k)$ if for every $m\in [N_1,N_2]$ such that $\|m\alpha_i\|<\epsilon$ for $i=1,...,k$, we have  $m\theta[1]\notin I$.}
\end{definition}

\begin{lemma}\label{lem} For every $l\geq 2$ there exists $L(l)\in \N$ such that for every $0<\epsilon<\frac{1}{2l^2}$, for every $v>0$, for every $k$, there exist $K(\epsilon)\in \N$ and $N'=N'(l,\epsilon,v,N,k)\in \N$ such that $\theta \in \mathcal{A}(N_1,N',I,\epsilon,k)$ 
for some interval $I$ of size $\frac{1}{l}$, implies that $\|\sum_{i=1}^kr_i\alpha_i-l'\theta\|<v$ for some $|r_1|,...,|r_k|<K(k,\epsilon)$ and some $|l'|<L(l)$.
\end{lemma}
The proof is a repetition of the proof of Lemma 1  in \cite{Fa-Tho}.  Instead of considering $\phi_\epsilon:\T\to\R$ one needs to consider $\phi^k_\epsilon:\T^k\to \R$. It follows by the proof that the number $L(l,\{\alpha_i\}_{i=1}^k)$ does not depend on the numbers $\{\alpha_i\}_{i=1}^k$ and that is why we just had $L(l)$ in the statement. Indeed, similarly to Lemma 1 in \cite{Fa-Tho}, one considers a polynomial $\varphi_l:\T\to\R$, $\varphi_l(y):=\sum_{0<|k|<L(l)}\hat{\varphi}_k e^{i2\pi ky}$, where $L(l)$ is such that 
\begin{itemize}
    \item $\varphi_l(y)>1$ for every $y\notin [0,\frac{1}{l}]$
	 \item $|\varphi_l(y)|<l^2$ for every $y\in \T$.
\end{itemize}
Therefore $L(l)\in \N$ does not depend on $\{\alpha_i\}_{i=1}^k$. 

\begin{remark}\label{den}{\em Consider an ergodic rotation $T:T^j\to T^j$, $T(x_1,...,x_j)=(x_1+\gamma_1,...,x_j+\gamma_j)$, for $\gamma_1,...,\gamma_j\in \T$. It follows that for every $k\in \N$ and every $\epsilon>0$, there exist (infinitely many) $m\in \N$ such that $\|m\gamma_i\|<\epsilon$ for $i=1,...,j$ and $k\nmid m$. Indeed, for every fixed $k\in \N$ there exist a sequence $(r_n)_{n\geq 1}$ such that $T^{r_n}(0)\to \frac{1}{k}T(0)$.}    
\end{remark} 
\begin{proposition}\label{inf} Fix rationally independent numbers $\{\alpha_i\}_{i\in \N}\in \T$. There exists a sequence $(s_n)$ such that $\lim_{n\to+\infty}\|s_n\alpha_i\|=0$ for $i=1,...$ and $(s_n\theta[1])_{n\geq 1}$ is dense in $\T$ if and only if $\theta\notin \Q+\Q\alpha_1+...$\footnote{By this we mean that there does not exist $n_0$ such that $\theta\in \Q+\Q\alpha_1+...+\Q\alpha_{n_0}$.}. 
\end{proposition}

\begin{proof} $ \ $

 We will use Lemma \ref{lem} for $k=1,2,...$.  Define for $n\geq 1$ the sequence $l_n=n+1$. 
 Let $\epsilon_n=\frac{1}{2(n+1)^2}$ and $K_n:=K(n,\epsilon_n)$.
  Define $v_n=\frac{1}{n}\inf_{0\leq |k_1|,...,|k_{n+1}|\leq K_{n+1}}\|\sum_{i=1}^{n+1}k_i\alpha_i\|$. Take $N_0=0$ and apply Lemma \ref{lem}  with $k=1,l=l_1,\epsilon=\epsilon_1, N=N_0,v=v_1$. Denote $N_1=N'(l_1,\epsilon_1,v_1,N_0,1)$. We apply Lemma \ref{lem}  inductively for $k=n, l=l_n, \epsilon=\epsilon_n,N=N_n,v=v_n$ and choose $N_{n+1}>N'(l_n,\epsilon_n,v_n,N_n,n)$ sufficiently large. Then we define an increasing sequence $(s_n)_{n=1}^{+\infty}$ by taking, for every $i\in \N$, all integers $s\in [N_i,N_{i+1}]$ such that $\|s\alpha_t\|<\epsilon_i$ for every $t=1,...,i$ (we can choose $N_{i+1}$ so that such $s\in [N_i,N_{i+1}]$ exists). Moreover by Remark \ref{den} we can choose $N_{i+1}$ so that for every $r=1,...,i$ there exists $s_r\in [N_i,N_{i+1}]$ with $r\nmid s_r$.
  
Notice first that for every $r\in \N$, $\lim_{n\to +\infty}\|s_n\alpha_r\|=0$.
Indeed, for every $j>r$ and every $t\in \N$ such that $s_t\in [N_j,N_{j+1}]$ we have $\|s_t\alpha_r\|<\epsilon_j$.

 Now, let $\theta\in \T$ be such that $s_n\theta[1]$ is not dense in $\T$. Then there exists $I\subset \T$, $|I|=\frac{1}{l_s}$ for some $s$ such that $s_n\theta[1]\notin I$. By the definition of the sequence $(s_n)_{n=1}^{+\infty}$ it follows that there exists $n_0$ such that $\theta\in \mathcal{A}(N_n,N_{n+1},I,\epsilon,n)$ for every $n\geq n_0$. Therefore, by Lemma \ref{lem} it follows that there are integers $k^n_1,...,k^n_n$ with $|k^n_i|<K_n$ for every $i=1,...,n$ such that
$\|\sum_{i=1}^nk^n_i\alpha_i-l'\theta\|<v_n$ for some $|l'|<L=L(l_s)$. Therefore, 
$\|\sum_{i=1}^nh^n_i\alpha_i-L!\theta\|<L!v_n$, for some numbers $h^n_1,...,h^n_n\in \N$ with $|h^n_i|<L!K_n$. It follows by triangle inequality that $\|\sum_{i=1}^{n+1}h^{n+1}_i\alpha_i-\sum_{i=1}^nh^n_i\alpha_i\|<L!v_n+L!v_{n+1}<2L!v_n$. By the definition of $v_n$, we get that there exists $n_1\in \N$ such that for $n\geq n_1$, these two combinations are equal. Therefore $|\sum_{i=1}^{n_1}h^{n_1}_i\alpha_i-L!\theta\|<L!v_n$ for every $n\geq n_1$. But $v_n\leq \frac{1}{n}\to 0$ and consequently $\theta\in \Q+\Q\alpha_1+...+\Q\alpha_{n_1}$.

On the other hand, it follows by construction of $(s_n)_{n\geq 1}$ that $(s_n\theta[1])_{n\geq 1}$ is not dense in $\T$  if $\theta\in \Q+\Q\alpha_1+...$.\end{proof}


\noindent {\it Proof of Theorem \ref{ext}.} \margem{I made some explanation.} For every $i \in \N$, let $\{s_n^{(i)}\}_{n\in \N}$ be a sequence as in Proposition \ref{inf} applied to the family of rationally independent numbers $\{\alpha_j\}_{j\in \N, j\neq i}\in \T$. Let $(N_s(i))_{s\geq 1}$ be the corresponding sequence of natural numbers given in the proof of Proposition \ref{inf}, that is  $\|s_t^{(i)} \alpha_r\|<\frac{1}{2(j+1)^2}$ for every $t\geq N_j(i)$ (this implies that $s_t>N_j(i)$) and every $r < j$.  Then define the sequence $\tilde{s}^{(i)}_n:=s^{(i)}_{n+N_i(i)}$

Then define $m_n$ to be the sequence $\tilde{s}_1^{(1)},\tilde{s}_2^{(1)},\tilde{s}_1^{(2)},\tilde{s}_3^{(1)},\tilde{s}_2^{(2)},\tilde{s}_1^{(3)},\tilde{s}_4^{(1)},\tilde{s}_3^{(2)},\tilde{s}_2^{(3)},\tilde{s}_1^{(4)},\ldots$. The sequence $m_n$ satisfies the conditions of Theorem \ref{ext}. Indeed, first note that for any irrational $\th$ there exists $i$ such that $\th \notin \bigcup_{i=1}^{+\infty}(\Q+...+\Q\alpha_{i-1}+\Q\alpha_{i+1}+...)$, hence $(m_n\theta[1])_{n\geq 1}$ is dense by just considering the subsequence $\tilde{s}_l^{(i)}$. 
Secondly fix $\epsilon>0$ and $k\in \N$. Let $r\in \N$ be such that $\frac{1}{2(r+1)^2}<\epsilon$. Define  $N_0:=(\max\{N_r(1),...,N_r(r)\})^2$.  Then, by definition of the sequence $(m_n)_{n\geq 1}$,
 $\|m_n\alpha_i\|<\frac{1}{2(r+1)^2}<\epsilon$, for every $n>N_0$ and every $i\in \{1,...,k\}$ except for at most one $i$ that satisfies $m_n=\tilde{s}_{l_n}^{(i)}$.  \hfill $\Box$

\begin{remark}\label{prem}\em{ It follows that for every $\epsilon>0, i\in \N$ there exist $n_0\in \N$ such that for every $n\geq n_0$, $\sum_{s=1}^i\|m_n\alpha_s\|<\frac{1}{2}+\epsilon$.
}
\end{remark}

\section{Proof of Theorem \ref{mes}.}

Fix rationally independent numbers $(\alpha_i)_{i\geq 1}\in \T$ and let $(m_n)_{n=1}^{+\infty}$ be the corresponding sequence given by Theorem \ref{ext}.\\
For the construction of the measure $\mu$ we will proceed similary to \cite{Fa-Tho} (and we borrow notation from there). For a probability measure $\nu$ on $\T$ we denote by $\nu^n=|\int_\T \|m_n\theta\| d\mu(\theta)|$. We will define inductively a sequence $(k_n)_{n\geq 1}$ so that the measure $\mu$ will be a weak limit of discrete measures $\mu_p:=\frac{1}{2^p}\sum_{i=1}^{2^p}\delta_{k_i\alpha_i}$ for some numbers $k_i\in \N$ such that there exists a sequence $(N_p)_{p=1}^{+\infty}$ for which 
\begin{enumerate}
    \item[(i)] For every $p\geq 1$, for every $j\in [1,p-1]$, for every $n\in [N_j,N_{j+1}]$, $\mu^n_p<\frac{1}{2^{j-1}}$ (for $j=0$ $\mu^n_p<1$).
	 \item[(ii)] For $p_0\in \N$ denote by $\eta_{p_0}=\frac{1}{4}\inf_{1\leq i<j\leq 2^{p_0}}\|k_i\alpha_i-k_j\alpha_j\|$. Then for every $l\in \N$ and every $r\in[1,2^{p_0}]$, $\|k_{l2^{p_0}+r}\alpha_{l2^{p_0}+r}-k_r\alpha_r\|<\eta_{p_0}$.
\end{enumerate}  
In fact, similarly to \cite{Fa-Tho}, we get that any weak limit $\mu$ of  a sequence  $\mu_p$ as above, satisfies the conclusion of Theorem \ref{mes}. Indeed, by (i) $\mu^n\to 0$.
By (ii) it follows that for each $p_0$, the intervals $I_r=[-\eta_{p_0}+k_r\alpha_r,\eta_{p_0}+k_r\alpha_r]$, $r=1,...,2^{p_0}$ are disjoint and $\mu_p(I_r)=\frac{1}{2^{p_0}}$ for every $p\geq p_0$ and therefore the limit measure $\mu$ is continuous. 

Therefore, we just have to construct the measures $\mu_p$ as in (i) and (ii). We will do an inductive construction, in which  we will additionally require that for every $p$
 \begin{equation}\label{add} \mu^n_p<\frac{1}{2^{p+1}}+\frac{1}{2^{p+3}} \;\;\text{for} \;n\geq N_p.
\end{equation}
 For $p=0$ let $k_1=1$, then $\mu$ is the Dirac measure at $\alpha_1$. Let $N_0=0$. For $p=1$, $k_2=1$, then $\mu_1$ is the average of Dirac measures at $\alpha_1$ and $\alpha_2$. We choose $N_1=1$.  This satisfies (i) and (\ref{add}) for $p=1$. \\
Assume that we have constructed $k_i$ for $i=1,...,2^p$, $N_l$ for $1\leq l\leq 2^p$ such that (i) and (\ref{add}) is satisfied up to $p$ and (ii) is satisfied for every $p_0\leq p$ and $0\leq l\leq 2^{p-p_0}-1$. We now choose $k_{2^p+1}$ so that $k_{2^p+1}\alpha_{2^p+1}$ is sufficiently close to $k_1\alpha_1$ so that 
$$\nu_{p,1}=\frac{1}{2^p} \sum_{i=1}^{2^p} \delta_{k_i\alpha_i}+\frac{1}{2^{p+1}}(\delta_{k_{2^p+1}}\alpha_{2^p+1}-\delta_{k_1\alpha_1})$$ satisfies $\nu^n_{p,1}<\frac{1}{2^{j-1}}$ for $n\in [N_j,N_{j+1}]$ and $j\in [0,p-1]$
($\nu^n_{p,1}=\mu_p^n+\frac{1}{2^{p+1}}(\|m_nk_{2^p+1}\alpha_{2^p+1}-m_nk_1\alpha_1\|)$). Moreover it follows that for $n\geq N_p$ we have $\nu^n_{p,1}<\mu^n_p+\frac{1}{2^{p+1}}<\frac{1}{2^{p+1}}+\frac{1}{2^{p+3}}+\frac{1}{2^{p+1}}<\frac{1}{2^{p-1}}$. Let $N_{p,1}>N_p$ be sufficiently large so that $\nu_{p,1}^n<\frac{1}{2^p}$ for $n\geq N_{p,1}$  ($\nu_{p,1}^n<\mu^n_p+\frac{1}{2^{p+1}}$ and $\mu^n_p$ can be arbitrary close to $\frac{1}{2^{p+1}}$ by Remark \ref{prem}). \\
Now construct idnuctively for $s=1,...,2^p$ the numbers $k_{2^p+s},N_{p,s}\in \N$ for the measures $\nu_{p,s}$ given by
$\nu_{p,s}=\mu_p+\frac{1}{2^{p+1}}(\sum_{i=1}^s(\delta_{k_{2^p+i}\alpha_{2^p+i}}-\delta_{k_i\alpha_i}))$.
It follows that by choosing $k_{2^p+s}$ so that $k_{2^p+s}\alpha_{2^p+s}$ is sufficiently close to $k_s\alpha_s$ and $N_{p,s}$ large enough, we can insure that

\begin{description}
    \item [A.]$\nu^n_{p,s}<\frac{1}{2^{j-1}}$ for every  $n\in [N_j,N_{j+1]}$, and $j\leq p-1$.
	 \item [B.]$\nu^n_{p,s}<\frac{1}{2^{p-1}}$ for $n\geq N_p$.
	 \item [C.]$\nu^n_{p,s}<\frac{1}{2^{p}}$ for $n\geq N_{p,s}$.
\end{description}
Indeed, for $s=1$ the above conditions are satisfied, assume that for some $s\geq 1$, they hold. We will prove that they hold for $s+1$. First note that $v_{p,s}-v_{p,s-1}=\frac{1}{2^{p+1}}(\delta_{k_{2^p+s}\alpha_{2^p+s}}-\delta_{k_s\alpha_s})$. Therefore by choosing $k_{2^p+s}$ so that $k_{2^p+s}\alpha_{2^p+s}$ is sufficienlty close to $k_s\alpha_s$ and by induction hypothesis, we get that $\nu^n_{p,s}<\frac{1}{2^{j-1}}$ for every $n\in [N_j,N_{j+1}]$ with $j\leq p-1$. 
The same arguments gives us $\nu^n_{p,s}<\frac{1}{2^{p-1}}$ for $N_{p,s-1}\geq n\geq N_p$. For $n>N_{p,s-1}$ we use the fact that $\nu^n_{p,s-1}<\frac{1}{2^p}$ to get $\nu^n_{p,s}<\nu^n_{p,s-1}+\frac{1}{2^{p+1}}<\frac{1}{2^p}+\frac{1}{2^p}=\frac{1}{2^{p-1}}$. 
For the third point we use the fact that for $n$ sufficiently large, $\|m_n\alpha_i\|$ is arbitrary small for all but one $i\in \{1,...,2^p+s\}$ (compare with Remark \ref{prem}), to get that for $N_{p,s}$ large enough, $\nu^n_{p,s}<\frac{1}{2^{p+1}}+\frac{1}{2^{p+2}}+\frac{1}{2^{p+2}}=\frac{1}{2^p}$, for $n\geq N_{p,s}$.
 
 Finally we define $\mu_{p+1}=\nu_{p,2^p}$ and observe that $\mu_{p+1}$ satisfies (i).  Moreover, by definition $\mu_{p+1}=\frac{1}{2^{p+1}}\sum_{i=1}^{2^{p+1}}\delta_{k_i\alpha_i}$ and using the properties of the sequence $(m_n)_{n\geq 1}$ ($\|m_n\alpha_i\|$ is arbitrary small for all but one $i=1,...,2^{p+1}$, see also Remark 7) we get that if $N_{p+1}$ is sufficiently large, then (\ref{add}) is satisfied for $\mu_{p+1}$.\\
\indent Moreover, for $l=2^{p-p_0}+l'-1$ we have $\|k_{l2^{p_0}+r}\alpha_{l2^{p_0}+r}-k_r\alpha_r\|\leq \|k_{l2^{p_0}+r}\alpha_{l2^{p_0}+r}-k_{l'2^{p_0}+r}\alpha_{l'2^{p_0}+r}\|+\|k_{l'2^{p_0}+r}\alpha_{l'2^{p_0}+r}-k_r\alpha_r\|<\eta_{p_0}$
By induction hypothesis and the choice of $k_{l2^{p_0}+r}$. Therefore (ii) is satisfied for $p+1$ and every $l\leq 2^{p+1}$. This finishes the proof. \hfill $\square$


\begin{thebibliography}{9}
\bibitem{Ad} T.\ Adams, {\em Tower multiplexing and slow weak mixing}, arXiv:1301.0791.
 \bibitem{BJLR} V. Bergelson, A. Del Junco, M. Lemanczyk, J. Rosenblatt, {\it   Rigidity and non-recurrence along sequences}, 	Ergodic Theory and Dynamical Systems, First View Article (2013), 1-39. 

\bibitem{Co-Fo-Si} I.P.\ Cornfield, S.V.\ Fomin, Ya.G. Sinai, {\em Ergodic Theory}, Springer-Verlag, New York, 1982.
\bibitem{Fa-Tho} B.\ Fayad, J-P.\ Thouvenot, {\em On the convergence to $0$ of $m_n\zeta[1]$.},  to appear in Acta Arithmetica. 
\end{thebibliography}
\end{document}